\documentclass[11pt,onecolumn]{article}
\usepackage{amsmath,amsthm,amsfonts}
\usepackage[bottom=1.3in]{geometry}             
\usepackage[numbers,sort]{natbib}
\usepackage{authblk}
\usepackage{graphicx}
\usepackage{fancyhdr}
\usepackage{amssymb}
\usepackage{amsmath}
\usepackage{epstopdf}
\usepackage[all,cmtip]{xy}
\usepackage[pagebackref,colorlinks,
    linkcolor={red!70!black},
    citecolor={green!60!black},
    urlcolor={blue!80!black}]{hyperref}
\usepackage[capitalize]{cleveref}
\usepackage{stmaryrd}
\usepackage{stackengine}
\usepackage{enumitem}
\usepackage{mathtools}
\usepackage{subcaption}
\usepackage{doi}
\usepackage{tikz}
\usepackage[textsize=footnotesize]{todonotes}
\usepackage{mathrsfs}
\usepackage{backref}
\usepackage[toctitles]{titlesec}

\theoremstyle{plain}
\newtheorem{theorem}{Theorem}[section]

\newtheorem{proposition}[theorem]{Proposition}
\theoremstyle{remark}

\theoremstyle{definition}

\newtheorem{example}[theorem]{Example}

\AtBeginEnvironment{example}{%
  \pushQED{\qed}%
}
\AtEndEnvironment{example}{\popQED\endexample}
\AtBeginEnvironment{remark}{%
  \pushQED{\qed}%
}
\AtEndEnvironment{remark}{\popQED\endexample}

\AtBeginEnvironment{definition}{%
  \pushQED{\qed}%
}
\AtEndEnvironment{definition}{\popQED\endexample}

\usepackage{makecell}

\usepackage{arydshln}

\title{A Note on the Carath\'{e}odory Number of the Joint Numerical Range}
 
  \author{Beatrice Maier}
 \author{Tim Netzer}

\affil{\small Department of Mathematics, University of Innsbruck, Austria}

\date{\today}                      
\begin{document}
\maketitle

\begin{abstract}
    We show that the Carath\'{e}odory number of the joint numerical range of $d$ many bounded self-adjoint operators is at most $d-1$, and even at most $d-2$ if the underlying Hilbert space has dimension at least $3$. This extension of the classical convexity results for numerical ranges shows that also joint numerical ranges are significantly less non-convex than general sets.
\end{abstract}

\section{Introduction and Preliminaries}

Let $T_1,\ldots, T_d\in\mathcal B(H)_{\rm sa}$ be self-adjoint bounded linear operators on the Hilbert space $H$. Their \emph{joint numerical range} is  defined as
$$\mathcal W(T_1,\ldots, T_d)\coloneqq \left\{ (\langle T_1h,h\rangle,\ldots, \langle T_dh,h\rangle)\mid h\in H, \Vert h\Vert=1\right\}\subseteq \mathbb R^d.$$
For $d\leqslant 2$ the joint numerical range  is  always convex,  a famous result by Toeplitz and Hausdorff \cite{hau,toe}. This is still true if $d=3$ and $\dim(H)\geqslant 3$ \cite{poon}, but fails for   $d\geqslant 4$ in general. 

An interesting measure for the non-convexity of a set $S\subseteq\mathbb R^d$ is its \emph{Carath\'{e}odory number} $\mathfrak c(S).$ It is defined as the smallest positive integer $n$, such that every element in the convex hull of $S$ is a convex combination of at most $n$ elements from $S$. A set is clearly convex if and only if its Carath\'eodory number is $1$. Carath\'eodory's Theorem \cite{car,st} states that ${\mathfrak c}(S)\leqslant d+1$ holds for any set $S\subseteq \mathbb R^d,$ and any set of $d+1$ affinely independent points shows that this general bound is optimal.
However, if $S\subseteq \mathbb R^d$ is connected, then ${\mathfrak c}(S)\leqslant d$ \cite{fen, bunt}.

As the image of the unit sphere in $H$ under a continuous map, $\mathcal W(T_1,\ldots, T_d)$ is clearly path-connected, and thus ${\mathfrak c}\left(\mathcal W(T_1,\ldots, T_d)\right)\leqslant d.$ In this note we observe that this can be improved  by one, and even by two if $\dim(H)\geqslant 3$ (\cref{thm_main}). This shows that joint numerical ranges are significantly less non-convex than general connected sets. Our proof is a combination of some of the standard methods to prove convexity of the classical numerical range, but suitably applied to joint numerical ranges.

\section{Main Result}

The proof of the following proposition contains  the initial idea from \cite{hau}, combined with \cite{lyu}, and in the multidimensional setup. It shows that joint numerical ranges are path-connected in a strong sense.
\begin{proposition}\label{prop:con} For any $T_1,\ldots, T_d\in{\mathcal B}(H)_{\rm sa}$ and for any affine hyperplane $P\subseteq \mathbb R^d$,  $$\mathcal W(T_1,\ldots, T_d)\cap P$$ is path-connected. If $\dim(H)\geqslant 3,$ this also holds when ${\rm codim}(P)=2.$
\end{proposition}
\begin{proof}For the affine hyperplane $P$ we can assume without loss of generality that $P=\{x_d=0\}.$ For $p_0,p_1\in \mathcal W(T_1,\ldots, T_d)\cap P$ there are unit vectors $h_0,h_1\in H$ with $$p_i=(\langle T_1h_i,h_i\rangle, \ldots, \underbrace{\langle T_dh_i,h_i\rangle}_{=0}).$$ Now as already shown and used in \cite{hau}, the set $$L\coloneqq \left\{h\in H\mid \Vert h\Vert=1, \langle T_dh,h\rangle=0\right\}$$ is path-connected. So there is a continuous path $h_t$ from $h_0$ to $h_1$ in $L,$ and $$p_t\coloneqq (\langle T_1h_t,h_t\rangle, \ldots, \underbrace{\langle T_dh_t,h_t\rangle}_{=0})$$ is thus a continuous path from $p_0$ to $p_1$ in  $\mathcal W(T_1,\ldots, T_d)\cap P.$

Now in the case ${\rm codim}(P)=2$ we can assume $P=\{x_{d-1}=0 \wedge x_d=0\}$.  We repeat the same argument, this times using that \begin{align*}L&= \left\{h\in H\mid \Vert h\Vert=1, \langle T_{d-1}h,h\rangle=\langle T_dh,h\rangle= 0\right\} \\ & =\{ h\in H\mid \Vert h\Vert =1, \langle(T_{d-1}+iT_d)h,h\rangle=0\}\end{align*} is path connected. This was proven in \cite{lyu} for the case that $\dim(H)\geqslant 3.$ 
\end{proof}

We can now prove our main result:
\begin{theorem}\label{thm_main}
For $d\geqslant 2$ and any $T_1,\ldots, T_d\in\mathcal B(H)_{\rm sa}$ we have ${\mathfrak c}\left(\mathcal W(T_1,\ldots, T_d)\right)\leqslant d-1.$ In case $d\geqslant 3$ and $\dim(H)\geqslant 3$ we even have  ${\mathfrak c}\left(\mathcal W(T_1,\ldots, T_d)\right)\leqslant d-2.$
\end{theorem}
\begin{proof}
Let $p\in{\rm conv}\left(\mathcal W(T_1,\ldots, T_d)\right).$ From \cite{fen, bunt} we already know  ${\mathfrak c}\left(\mathcal W(T_1,\ldots, T_d)\right)\leqslant d,$ thus  $p$ is a convex combination of $d$ points from $\mathcal W(T_1,\ldots, T_d).$ Let $P$ be an affine hyperplane in $\mathbb R^d$ containing all these points. Then $p$ is already in the convex hull of $\mathcal W(T_1,\ldots, T_d)\cap P$. Since $\mathcal W(T_1,\ldots, T_d)\cap P$ is connected (\Cref{prop:con}) and lives in the $(d-1)$-dimensional space $P$, we obtain $\mathfrak c\left(\mathcal W(T_1,\ldots, T_d)\cap P\right)\leqslant d-1,$ again from \cite{fen, bunt}. Thus $p$ is a convex combination of at most $d-1$ points from $\mathcal W(T_1,\ldots, T_d),$ which proves the first claim.

In case $\dim(H)\geqslant 3$ we can iterate this argument once. Write $p$ as a convex combination of $d-1$ many points from $\mathcal W(T_1,\ldots T_d).$ These points live in an affine subspace $P$ of dimension $d-2,$ and since $\mathcal W(T_1,\ldots, T_d)\cap P$ is still connected by \Cref{prop:con}, \cite{fen, bunt} again implies that $p$ is a convex combination  of at most $d-2$ elements from $\mathcal W(T_1,\ldots, T_d)$.
\end{proof}

\bibliographystyle{abbrvnat}
\bibliography{references}
\end{document}